\documentclass[a4paper, 11pt,reqno]{amsart}

\usepackage[T1]{fontenc}      % Uses 'T1' font encoding.
\usepackage[foot]{amsaddr}    % Manages options for displayed author information.

\usepackage[british]{babel}   % Support for different languages.
\usepackage{csquotes}         % Adds quotation marks according to language.
\usepackage[babel]{microtype} % Makes things look nicer (spacing and so on).

\usepackage{mathtools}        % Extension of 'amsmath' which fixes some bugs.
\usepackage{amssymb}          % Extra symbols and fonts.
\usepackage{amsthm}           % Helps to define "theorem-like" environments.
\usepackage{amstext}          % Defines a '\text' macro for math environments.
\usepackage{mleftright}       % Fixes spacing issue.
\usepackage{gensymb}          % Allows to use degrees.
\usepackage[makeroom]{cancel} % Defines cancellation of terms in formulae.
\usepackage{mathdots}         % Adds some definitions of dots.
\usepackage{mathrsfs}         % Adds rsfs fonts.
\usepackage{stickstootext}
\usepackage[stickstoo,varbb]{newtxmath} % These two lines are an alternative for stix2 which do essentially the same but fix several issues with some symbols...
\makeatletter
\DeclareFontFamily{U}{ntxmia}{\skewchar \font =127}
 \DeclareFontShape{U}{ntxmia}{m}{it}{
                        <-> \ntxmath@scaled ntxmia
                      }{}    
                      \DeclareFontShape{U}{ntxmia}{b}{it}{
                        <-> \ntxmath@scaled ntxbmia
                      }{}
\makeatother

\usepackage{graphicx}         % Provides commands to include graphics.
\usepackage{psfrag}           % Gives some extra useful things for graphics.
\usepackage{caption}          % Options for formatting float captions.
\usepackage{tikz}             % Language to create graphics programmatically.
\usetikzlibrary{backgrounds}

\usepackage{booktabs}         % Enhances the quality of tables.

\usepackage{enumitem}         % Allows customization for enumerations.

\usepackage[margin=1in]{geometry} % Provides commands to define the page layout.

\usepackage[textsize=footnotesize]{todonotes}
\usepackage[numbers,sort]{natbib}

\makeatletter
\def\NAT@spacechar{~}% NEW
\makeatother

\usepackage{hyperref}              % For hyperlinks within the document
\usepackage{cleveref}  % For improved referencing options; use [capitalise] if desired
\hypersetup{colorlinks=true,
   citecolor=blue,
   filecolor=blue,
   linkcolor=blue,
   urlcolor=blue
}
\usepackage{url}

\makeatletter
\if@cref@capitalise
\crefname{figure}{figure}{figures}
\crefname{claim}{Claim}{Claims}
\else
\crefname{figure}{Figure}{Figures}
\crefname{claim}{claim}{claims}
\fi
\makeatother
\Crefname{figure}{Figure}{Figures}
\Crefname{claim}{Claim}{Claims}

\usepackage{algorithm}
\usepackage{algpseudocode}

\allowdisplaybreaks

\newtheorem{definition}{Definition}
\newtheorem{claim}{Claim}

\newtheorem{theorem}[definition]{Theorem}

\newtheorem{lemma}[definition]{Lemma}

\newtheorem{conjecture}[definition]{Conjecture}

\newtheorem*{nocremark}{Remark}

\newenvironment{claimproof}{%
\let\origqed=\qedsymbol%
\renewcommand{\qedsymbol}{$\blacktriangleleft$}%
\begin{proof}}{\end{proof}\let\qedsymbol=\origqed}

\numberwithin{equation}{section}

\renewcommand{\binom}[2]{\ensuremath{\mleft(\kern-.1em\genfrac{}{}{0pt}{}{#1}{#2}\kern-.1em\mright)}}    % This makes binomial numbers nicer with stix2 (in displayed equations). Remove if stix2 is not loaded.
\newcommand{\inbinom}[2]{\ensuremath{\bigl(\kern-.1em\genfrac{}{}{0pt}{}{#1}{#2}\kern-.1em\bigr)}} % This is better for inline equations, as it will keep sizes of parentheses consistent and not create extra vertical space.

\newcommand*\nume{\ensuremath{\mathrm{e}}}

\makeatletter
\def\moverlay{\mathpalette\mov@rlay}
\def\mov@rlay#1#2{\leavevmode\vtop{%
  \baselineskip\z@skip \lineskiplimit-\maxdimen
  \ialign{\hfil$\m@th#1##$\hfil\cr#2\crcr}}}
\newcommand{\charfusion}[3][\mathord]{
    #1{\ifx#1\mathop\vphantom{#2}\fi
        \mathpalette\mov@rlay{#2\cr#3}
      }
    \ifx#1\mathop\expandafter\displaylimits\fi}
\makeatother

\newcommand{\bigO}{\ensuremath{\mathcal{O}}}

\newcommand{\COMMENT}[1]{}
\title{Hamiltonicity of graphs perturbed by a random geometric graph}

\author[A.~Espuny D\'iaz]{Alberto Espuny D\'iaz}
\email{alberto.espuny-diaz@tu-ilmenau.de}

\address{Institut f\"ur Mathematik, Technische Universit\"at Ilmenau, 98684 Ilmenau, Germany.}

\thanks{This research has been partially supported by the Carl Zeiss Foundation.}

\date{\today}
\begin{document}

\begin{abstract}
We study Hamiltonicity in graphs obtained as the union of a deterministic \mbox{$n$-vertex} graph $H$ with linear degrees and a $d$-dimensional random geometric graph $G^d(n,r)$, for any $d\geq1$.
We obtain an asymptotically optimal bound on the minimum $r$ for which a.a.s.~$H\cup G^d(n,r)$ is Hamiltonian.
Our proof provides a linear time algorithm to find a Hamilton cycle in such graphs.

% \vspace{1em}
% \noindent
% Keywords: Hamiltonicity, random geometric graphs, randomly perturbed graphs, absorption.
\end{abstract}

\maketitle

\section{Introduction}

Randomly perturbed graphs are one of the thriving areas in the study of random combinatorial structures, and many interesting results in this field have been proved in recent years.
The main general goal in this area is to study properties of graphs which are obtained as the union of a deterministic graph $H$ (usually with a minimum degree condition) and a random graph $G$, particularly when $H$ does not satisfy the property of interest and $G$ is unlikely to.

Research in this direction sparked off with the work of \citet{BFM03}, who studied Hamiltonicity (that is, the property of containing a cycle which covers every vertex of a graph) in $H\cup G_{n,p}$, where $H$ is an $n$-vertex graph with minimum degree at least $\alpha n$, for some $\alpha\in(0,1/2)$ (if $\alpha\geq1/2$, then Dirac's theorem guarantees that $H$ contains a Hamilton cycle), and $G_{n,p}$ is the binomial random graph where each pair of vertices forms an edge independently with probability~$p$.
They showed that, whenever $p\geq C(\alpha)/n$, asymptotically almost surely (a.a.s.) $H\cup G_{n,p}$ is Hamiltonian, improving on the threshold for Hamiltonicity in $G_{n,p}$ by a logarithmic factor.
Since then, Hamiltonicity has also been considered in randomly perturbed directed graphs~\cite{BFM03,KKS16}, hypergraphs~\cite{KKS16,MM18,HZ20} and subgraphs of the hypercube~\cite{CEGKO20}.
Many other properties have been considered as well (e.g., powers of Hamilton cycles~\cite{BMPP20,DRRS20,ADRRS21,NT21}, \mbox{$F$-factors}~\cite{BTW19,HMT19+,BPSS20,BPSS21}, spanning trees~\cite{KKS17,BHKMPP19,JK19} or general bounded degree spanning graphs~\cite{BMPP20}), and in most cases significant improvements on the probability threshold have been achieved.
To the best of our knowledge, all of these results consider (hyper/di)graphs perturbed by a binomial random structure, such as $G_{n,p}$, or its $G_{n,m}$ counterpart.
Only very recently, \citet{EG21} considered Hamiltonicity in graphs perturbed by a random regular graph.

In this paper we consider graphs perturbed by a random geometric graph.
We consider (labelled) random geometric graphs in $d$ dimensions, which are defined as follows.
Let $V\coloneqq\{1,\ldots,n\}$, and let $X_1,\ldots,X_n$ be $n$ independent uniform random variables on $[0,1]^d$.
Then, consider a positive real number $r$ and let $E\coloneqq\{\{i,j\}:\lVert X_i-X_j\rVert\leq r\}$, where $\lVert\cdot\rVert$ denotes the Euclidean norm.
The resulting graph $(V,E)$ is denoted by $G^d(n,r)$.

Hamitonicity of random geometric graphs is fairly well understood.
In dimension $2$, Díaz, Mitsche and Pérez-Giménez \cite{DMP07} determined that the sharp threshold for Hamiltonicity is \mbox{$r^*\!\coloneqq\!(\log n/(\pi n))^{1/2}$} (which means that, for all $\epsilon>0$, if $r\geq(1+\epsilon)r^*$, then a.a.s.~$G^2(n,r)$ is Hamiltonian, and if $r\leq(1-\epsilon)r^*$, then a.a.s.~$G^2(n,r)$ is not Hamiltonian), where $\log$ stands for the natural logarithm.
This result was later strengthened by \citet{BBKMW11} and, independently, \citet{MPW11}, who also extended their results to higher dimensions, proving that, for each positive integer $d\geq2$, there exists a constant $c_d$ such that the sharp threshold for Hamiltonicity in $G^d(n,r)$ is $(c_d\log n/n)^{1/d}$.
(These two papers actually prove a stronger ``hitting time'' result.)
The same statement is true for $d=1$, which is a special case (see~\cite[Section~13.1]{Pen03} and the references therein, together with a remark in the introduction of~\cite{BBKMW11}).
All of these results also extend to other $\ell_p$ norms, $1\leq p\leq\infty$, the only difference being the value of the constant $c_d=c(d,p)$.
For more information about random geometric graphs, we recommend the works of \citet{Pen03,Pen16}.

Our goal here is to prove the following result.

\begin{theorem}\label{thm:main}
For every integer $d\geq1$ and $\alpha\in(0,1/2)$, there exists a constant $C$ such that the following holds for any function $r=r(n)$ with $r\geq(C/n)^{1/d}$.
Let $H$ be an $n$-vertex graph with minimum degree at least $\alpha n$.
Then, a.a.s.\ $H\cup G^d(n,r)$ is Hamiltonian.
\end{theorem}

This result is asymptotically best possible, up to the constant factor (which we will make no effort to optimise).
The rest of the paper is organised as follows.
In \cref{section2}, we set our notation and also introduce a probabilistic tool which will be important for our proof.
We then prove \cref{thm:main} in \cref{section3}.
Finally, in \cref{section4} we make several observations about our result, as well as some extensions that follow from its proof.

\section{Preliminaries}\label{section2}

\subsection{Notation}

For any integer $n$, we write $[n]\coloneqq\{i\in\mathbb{Z}:1\leq i\leq n\}$.
If we write parameters in a hierarchy, we assume they are chosen from right to left.
To be more precise, whenever we write $0<a\ll b\leq 1$, we mean that there exists an unspecified, non-decreasing function $f\colon\mathbb{R}\to\mathbb{R}$ such that the ensuing claim holds for all $0<b\leq 1$ and for all $0<a\leq f(b)$.
This can be immediately generalised to longer hierarchies, and also to hierarchies where one parameter may depend on two or more other parameters.
We say that a sequence of events $\{\mathcal{E}_i\}_{i\in\mathbb{N}}$ holds \emph{asymptotically almost surely} (a.a.s.) if $\mathbb{P}[\mathcal{E}_i]\to1$ as $i\to\infty$.
In all asymptotic statements, we will ignore rounding issues whenever these do not affect the arguments.

Most of our graph theoretic notation is standard.
Given a graph $G$, we use $V(G)$ and $E(G)$ to denote its vertex set and edge set, respectively.
We always consider labelled graphs, meaning that whenever we say that $G$ is an $n$-vertex graph we may implicitly assume that $V(G)=[n]$.
If $G$ is a geometric graph (meaning here that each of its vertices is assigned to a position in $\mathbb{R}^d$ for some integer $d$), then $V(G)$ may interchangeably be used to refer to the set of positions to which the vertices of $G$ are assigned, and similarly the notation $v$ may refer to a vertex or to its position.
If $e=\{u,v\}$ is an edge, we usually abbreviate it as $e=uv$.
Given any vertex $v\in V(G)$, we define $N_G(v)\coloneqq\{u\in V(G):uv\in E(G)\}$, and $d_G(v)\coloneqq|N_G(v)|$ is its \emph{degree}.
We denote the minimum and maximum vertex degrees of $G$ by $\delta(G)$ and $\Delta(G)$, respectively.
Given a graph $G$ and a set of vertices $A\subseteq V(G)$, we denote by $G[A]$ the graph on vertex set $A$ whose edges are all edges of $G$ which have both endpoints in $A$.
A \emph{path} $P$ is a graph whose vertices can be labelled in such a way that $E(P)=\{v_iv_{i+1}:i\in[|V(P)|-1]\}$.
If the endpoints of a path (the first and last vertices in the labelling described above) are~$u$ and~$v$, we sometimes refer to it as a $(u,v)$-path.
Given a $(u,v)$-path~$P$ and a $(v,w)$-path~$P'$ such that $V(P)\cap V(P')=\{v\}$, we write~$PP'$ to denote the path obtained by concatenating $P$ and $P'$ (formally, this is the union graph of~$P$ and~$P'$).
If $P'$ is a single edge $vw$, we will write this as $Pw$.
Multiple concatenations will be written in the same way.

\subsection{Azuma's inequality}

Let $\Omega$ be an arbitrary set (in our case, we will take $\Omega=[0,1]^d$), and let $f\colon \Omega^n\to\mathbb{R}$ be some function.
We say that $f$ is $L$-Lipschitz, for some positive $L\in\mathbb{R}$, if, for all $x,y\in\Omega^n$ such that $x$ and $y$ are identical in all but one coordinate, we have that $|f(x)-f(y)|\leq L$.
The following consequence of Azuma's inequality (see, for instance, the book of \citet[Corollary~2.27]{JLR}) will be useful for bounding the deviations of certain random variables.

\begin{lemma}\label{lem:Azuma}
Let $X_1,\ldots,X_n$ be independent random variables taking values in a set $\Omega$, and let $f\colon\Omega^n\to\mathbb{R}$ be an $L$-Lipschitz function.
Then, for any $t\geq0$, the random variable $X\coloneqq f(X_1,\ldots,X_n)$ satisfies that
\[\mathbb{P}[X\geq\mathbb{E}[X]+t]\leq \nume^{-\frac{t^2}{2L^2n}}\qquad\text{ and }\qquad\mathbb{P}[X\leq\mathbb{E}[X]-t]\leq \nume^{-\frac{t^2}{2L^2n}}.\]
\end{lemma}

\section{Proof of \texorpdfstring{\cref{thm:main}}{Theorem \ref{thm:main}}}\label{section3}

Let $0<1/n_0\ll1/C\ll1/d,\alpha\leq1$, where $d\in\mathbb{N}$\COMMENT{We may also assume $\alpha<1/2$, as otherwise, the result holds by Dirac's theorem.}.
Throughout, we assume that $n\geq n_0$ is a natural number.
Let $r\coloneqq(C/n)^{1/d}$, $s\coloneqq1/\lceil2\sqrt{d}/r\rceil$ and $K\coloneqq s^dn$.
Our choice of~$C$ ensures that~$K$ is sufficiently large for all the ensuing claims to hold.
(On an intuitive level, we may replace the hierarchy above by $0<1/n\ll1/K\ll1/d,\alpha\leq1$, and may think of $K$ as essentially interchangeable with $C$.)
Tessellate $[0,1]^d$ using \mbox{$d$-dimensional} hypercubes of side~$s$\COMMENT{So there are $s^{-d}$ cells.} (intuitively, $s$ is close to $r/(2\sqrt{d})$, but possibly slightly smaller to guarantee that the tessellation above exists).
We refer to each of the smaller $d$-dimensional hypercubes as a \emph{cell}, and denote the set of all cells as $\mathcal{C}$.
Observe that $|\mathcal{C}|=s^{-d}=n/K$.
We say that two cells $c_1,c_2\in\mathcal{C}$ are \emph{friends} if their boundaries intersect (in particular, if the boundaries share a single point, they count as intersecting).
It follows that each cell is friends with at most $3^d-1$ other cells.
Given any set of points $S\subseteq[0,1]^d$, we say that a cell is \emph{sparse} with respect to~$S$ if it contains at most~$R\coloneqq2\cdot3^d$ points, and we call it \emph{dense} with respect to~$S$ otherwise.

Consider a labelling of the vertices of $V(H)$ as $v_1,\ldots,v_n$, and let $X_1,\ldots,X_n$ be independent uniform random variables on $[0,1]^d$.
We may assume that none of the variables take their values in the boundary of any of the cells, as this is an event that occurs with probability $1$.
Consider the random geometric graph $G=G^d(n,r)$ on the vertex set of $H$ obtained by assigning position $X_i$ to $v_i$.
Note that, by the definition of~$s$, if some $v\in V(H)$ lies in a cell $c\in\mathcal{C}$, then it is joined by an edge of $G$ to all other vertices in~$c$, as well as to all vertices in cells which are friends of~$c$\COMMENT{Given a fixed cell, its friends are those which are incident to it (i.e., their boundaries intersect in some way, corners are enough).
Thus, it suffices to have that two diagonals of cells are at most $r$ (the worst case would be given when we want to see if a vertex in a corner of a cell is connected to a vertex in the furthest corner of a friend which is as far as possible, that is, it only shares a corner with a cell that shares a corner with the first cell).
The diagonal of a cell has length $\sqrt{d}s$, so we want to enforce that $2\sqrt{d}s\leq r$, and this follows by the choice of $s$.}.

Let us provide here a brief sketch of the proof.
When considering the random geometric graph $G^d(n,r)$, one can show that a.a.s.\ ``most'' cells will be dense with respect to the vertex set, and only a small proportion of them will be sparse.
Now remove all sparse cells from consideration and define an auxiliary graph $\Gamma$, whose vertex set is the set of remaining cells, and where two cells are joined by an edge if they are friends.
A simple algorithm allows us to construct a cycle containing all the vertices which lie in cells that form a connected component in $\Gamma$.
Indeed, fix one such component and consider a spanning subtree $T$.
By definition, $T$ has bounded degree.
One can then ``walk'' along the edges of $T$ to visit all cells, using each edge of $T$ exactly twice, and end in the starting cell.
While doing this, and since the edges of $T$ correspond to cells which are friends, one can incorporate the vertices in the cells into a path (where a suitable number of vertices must be chosen each time a cell is visited).
Eventually, this path gets closed into a cycle when the walk returns to the starting cell.
The bound on the maximum degree of $T$ and the definition of sparse cells (with respect to $V(H)$) are crucial for this process to work.
This simple idea has been at the heart of many results about Hamiltonicity in random geometric graphs~\cite{DMP07,BBKMW11,MPW11}.
Since ``most'' cells are dense and we can apply the previous algorithm to each component of $\Gamma$, we can obtain a set of cycles covering almost all vertices.
In order to obtain a Hamilton cycle, all that remains is to join the cycles into a single, longer cycle and to incorporate the missing vertices.
It is for this purpose that we will need to use the edges of $H$ (and also be somewhat careful when choosing the cycles above).
The following definition will be crucial to achieve our goal.

Given any pair of (not necessarily distinct) vertices $u,v\in V(H)$, we say that a cell is \emph{$\{u,v\}$-linked} (with respect to~$H$) if it contains two distinct vertices $x,y\in V(H)\setminus\{u,v\}$ such that $ux,vy\in E(H)$; if it does not contain such a pair, we say that it is \emph{$\{u,v\}$-unlinked}.
In the particular case when $u=v$, we will say that the cell is $v$-linked/unlinked to mean that it is $\{v,v\}$-linked/unlinked.
One can easily see an intuitive reason why this notion is useful.
Say that a cell $c$ is sparse with respect to $V(H)$ and contains two vertices $u$ and $v$, and say we find a $\{u,v\}$-linked cell $c'$ which is dense with respect to $V(H)$.
Then, when ``walking'' through the component of $\Gamma$ containing $c'$, we may actually leave this component to reach $u$, incorporate all vertices in $c$ into the path we are constructing, and then go back from $v$ to $c'$.
This will allow us to incorporate all vertices in sparse cells into the cycles we constructed above.
A similar approach will allow us to join the different cycles.

We now proceed to formally prove \cref{thm:main}.
We begin with the following claim.

\begin{claim}\label{claim:fewsparse}
The following properties hold a.a.s.:
\begin{enumerate}[label=$(\mathrm{\roman*})$]
    \item\label{claim:fewsparse1} The number of cells which are sparse with respect to $V(H)$ is at most $\nume^{-K/2}n$.
    \item\label{claim:fewsparse2} For each pair of (not necessarily distinct) vertices $u,v\in V(H)$, the number of $\{u,v\}$-unlinked cells is at most $4\nume^{-\alpha K/4}n$.
\end{enumerate}
\end{claim}

\begin{claimproof}
\ref{claim:fewsparse1}
Let $f\colon([0,1]^d)^n\to\mathbb{Z}_{\geq0}$ be a function which, given a set $S$ of points $x_1,\ldots,x_n\in[0,1]^d$, returns the number of cells which are sparse with respect to $S$.
Note that $f$ is a $1$-Lipschitz function.

For each $v\in V(H)$ and each cell $c$, we have that $\mathbb{P}[v\in c]=s^d$.
Thus, since the variables $X_i$ are independent, for a fixed cell $c$ we have that\COMMENT{Let $R\coloneqq2\cdot3^d$. We have
\begin{align*}
    \mathbb{P}[c\text{ is sparse}]&=\sum_{i=0}^R\binom{n}{i}s^{di}(1-s^d)^{n-i}\leq (R+1)\binom{n}{R}s^{dR}(1-s^d)^{n-R}\leq (R+1)\frac{n^R}{R!}\frac{K^R}{n^R}\left(1-\frac{K}{n}\right)^{n-R}\\
    &\leq2R\frac{K^R}{R!}\left(1-\frac{K}{n}\right)^n\leq\frac{2K^R}{(R-1)!}\nume^{-K}\leq K^R\nume^{-K},
\end{align*}
where the first inequality holds since $1/K\ll1/d$, which ensures that the function $\inbinom{n}{x}s^{dx}(1-s^d)^{n-x}$ is increasing in this range; the third holds since $1/n_0\ll1/K\ll1/d$ (so $(1-K/n)^R\geq3/4$), and the next holds by the standard $(1-c/x)^x\leq \nume^{-c}$.}
\begin{align*}
    \mathbb{P}[c\text{ is sparse with respect to }V(H)]&=\sum_{i=0}^{R}\binom{n}{i}s^{di}(1-s^d)^{n-i}\leq(R+1)\binom{n}{R}s^{dR}(1-s^d)^{n-R}\\
    &\leq2R\frac{K^R}{R!}\left(1-\frac{K}{n}\right)^n\leq K^{R}\nume^{-K}.
\end{align*}
(Here the first and third inequalities rely on the hierarchy we established at the beginning of the proof; in particular, we may think of $K$ as being much larger than $R$.
We also use the fact that $s^d=K/n$.)
Let $Y\coloneqq f(X_1,\ldots,X_n)$ be the number of cells which are sparse with respect to $V(H)$, so $\mathbb{E}[Y]\leq K^{R-1}\nume^{-K}n$\COMMENT{Recall the number of cells is $s^{-d}=n/K$.}.
Since $f$ is \mbox{$1$-Lipschitz}, it follows by \cref{lem:Azuma} that\COMMENT{Let $R\coloneqq2\cdot3^d$. The first inequality holds for sufficiently large $K$, as $K^R\ll \nume^{K/2}$. We then have that
\begin{align*}
    \mathbb{P}\left[Y\geq2K^{R-1}\nume^{-K}n\right]&=\mathbb{P}\left[Y\geq K^{R-1}\nume^{-K}n+K^{R-1}\nume^{-K}n\right]\\
    &\leq\mathbb{P}\left[Y\geq\mathbb{E}[Y]+K^{R-1}\nume^{-K}n\right]\leq \nume^{-K^{2R-2}\nume^{-2K}n/2}=\nume^{-\Theta(n)}.
\end{align*}
}
\[\mathbb{P}[Y\geq \nume^{-K/2}n]\leq\mathbb{P}[Y\geq2K^{R-1}\nume^{-K}n]\leq \nume^{-\Theta(n)}.\]

\ref{claim:fewsparse2}
We proceed in a similar way.
For each positive integer $m$, let $g_m\colon([0,1]^d)^m\to\mathbb{Z}_{\geq0}$ be a function which, given a set $S$ of $m$ points $x_1,\ldots,x_m\in[0,1]^d$, returns the number of cells which contain at most one point.
Clearly, $g_m$ is $1$-Lipschitz for every $m\in\mathbb{N}$.
Given any set $V\subseteq V(H)$, we say that a cell is \emph{$V$-vacant} if it contains at most one of the vertices of~$V$.

Fix two (not necessarily distinct) vertices $u,v\in V(H)$, and let $Z$ be the number of $\{u,v\}$-unlinked cells.
Let $N\coloneqq N_H(u)\cap N_H(v)$.
We split the analysis into two cases.

Assume first that $\ell\coloneqq|N|\geq\alpha n/2$.
Let $i_1,\ldots,i_\ell$ be the indices of the vertices of $H$ which lie in~$N$.
It follows that, for a fixed cell $c$\COMMENT{We have that \begin{align*}
    \mathbb{P}[c\text{ is }N\text{-vacant}]&=(1-s^d)^\ell+\ell s^d(1-s^d)^{\ell-1}=\left(1-\frac{K}{n}\right)^\ell+\ell \frac{K}{n}\left(1-\frac{K}{n}\right)^{\ell-1}\\
    &\leq\left(1-\frac{K}{n}\right)^{\alpha n/2}+\frac{\alpha K}{2}\left(1-\frac{K}{n}\right)^{\alpha n/2-1}\leq\left(1+\alpha K\right)\left(1-\frac{K}{n}\right)^{\alpha n/2}\leq 2K\nume^{-\alpha K/2}\leq K\nume^{-\alpha K/4},
\end{align*}
where the first inequality can be proved easily (as $\ell$ is linear, the exponents play the largest role), in the second we use the trivial bound $1-K/n\geq1/2$, and the third uses the usual bound on $(1-1/x)^x$.}\COMMENT{Notice that here we need $\alpha K/4\geq\log K$, as otherwise this does not give us much information.
In particular, we need $K\gg\alpha^{-1}$.
This prevents us from extending our approach (in any meaningful way) to values of $\alpha$ tending to $0$ with $n$.
For instance, the probability that a fixed cell is $\{u,v\}$-unlinked tends to $1$ if, say, $\alpha=\theta(\sqrt{n})$ and we want to take $K$ as in the conjecture, so most cells are bad for us.},
\[\mathbb{P}[c\text{ is }N\text{-vacant}]=(1-s^d)^\ell+\ell s^d(1-s^d)^{\ell-1}\leq K\nume^{-\alpha K/4}.\]
Let $Z'\coloneqq g_\ell(X_{i_1},\ldots,X_{i_\ell})$, so $\mathbb{E}[Z']\leq \nume^{-\alpha K/4}n$.
Note that every $\{u,v\}$-unlinked cell is $N$-vacant, so $Z\leq Z'$.
Since $g_\ell$ is $1$-Lipschitz, by \cref{lem:Azuma} we conclude that\COMMENT{We have that
\begin{align*}
    \mathbb{P}\left[Z\geq2\nume^{-\alpha K/4}n\right]&\leq\mathbb{P}\left[Z'\geq2\nume^{-\alpha K/4}n\right]=\mathbb{P}\left[Z'\geq \nume^{-\alpha K/4}n+\nume^{-\alpha K/4}n\right]\\
    &\leq\mathbb{P}\left[Z'\geq\mathbb{E}[Z']+\nume^{-\alpha K/4}n\right]\leq\exp\left(-\nume^{-\alpha K/2}n/2\right)=\nume^{-\Theta(n)}.
\end{align*}
}
\[\mathbb{P}[Z\geq2\nume^{-\alpha K/4}n]\leq \nume^{-\Theta(n)}.\]

Assume now that $\ell<\alpha n/2$ and let $S\coloneqq N_H(u)\setminus(N_H(v)\cup\{v\})$ and $T\coloneqq N_H(v)\setminus(N_H(u)\cup\{u\})$.
Note that $|S|,|T|\geq\alpha n/2$.
By following the same argument as above, we have that, with probability at least $1-\nume^{-\Theta(n)}$, the number of $S$-vacant cells is at most $2\nume^{-\alpha K/4}n$.
The same holds for the number of $T$-vacant cells.
Note that every $\{u,v\}$-unlinked cell must be $S$-vacant or $T$-vacant.
We thus conclude that 
\[\mathbb{P}[Z\geq4\nume^{-\alpha K/4}n]\leq \nume^{-\Theta(n)}.\]

Finally, the statement holds by a union bound over all pairs of vertices $\{u,v\}\subseteq V(H)$.
\end{claimproof}

Condition on the event that $G$ satisfies the properties of the statement of \cref{claim:fewsparse}, which holds a.a.s.
Let $\mathcal{C}_{\mathrm{s}}$ be the set of cells which are sparse with respect to $V(H)$ (to which we simply refer as ``sparse cells'' from now on), and let $\mathcal{C}_{\mathrm{d}}\coloneqq\mathcal{C}\setminus\mathcal{C}_{\mathrm{s}}$.
We define an auxiliary graph $\Gamma$ with vertex set $\mathcal{C}_{\mathrm{d}}$ where two cells are joined by an edge whenever they are friends.
In particular, $\Delta(\Gamma)\leq 3^d-1$.

\begin{claim}\label{claim:fewcomponents}
The number of connected components of\/ $\Gamma$ is at most $\nume^{-K/3}n$.
\end{claim}

\begin{claimproof}
By \cref{claim:fewsparse}\ref{claim:fewsparse1}, there are at most $\nume^{-K/2}n$ sparse cells.
Recall that the total number of cells is $n/K$.
Thus, the number of components of $\Gamma$ of size at least $n^{1/d}/(4K^{1/d})$ is at most $4(n/K)^{1-1/d}\leq\nume^{-K/3}n/2$. %, where the last inequality holds by taking $n$ sufficiently large.

Now consider any component $\gamma$ of $\Gamma$ with fewer than $n^{1/d}/(4K^{1/d})$ cells.
We claim that the number of sparse cells which are friends with some cell of $\gamma$ is at least $2^d-1$.
Indeed, choose an arbitrary cell $c_0\in V(\gamma)$, and choose a corner of the hypercube $[0,1]^d$ which is at distance at least $1/3$ from $c_0$ in each coordinate direction.
Let us assume, without loss of generality, that said corner is $\{1\}^d$.
Now, for $i\geq0$, we follow an iterative process.
Consider the cube $C_i$ of side length $2s$ having $c_i$ in a corner and growing in all directions towards $\{1\}^d$; this is a union of $2^d$ cells, all of which are friends with each other.
If $C_i$ only contains sparse cells (other than $c_i$), we are done (and stop the iterative process), so assume otherwise.
Then, there must be some dense cell $c_{i+1}\subseteq C_i\setminus c_i$.
This new cell $c_{i+1}$, being a friend of $c_i$, also lies in $\gamma$, and is a translation of $c_i$ by a vector all whose coordinates are non-negative and bounded from above by $s$.
This process keeps going until either we find an index $i$ such that $C_i$ contains $2^d-1$ sparse cells or we reach one of the sides of the hypercube $[0,1]^d$.
However, reaching the side requires at least $1/(3s)=n^{1/d}/(3K^{1/d})$ iterations, which contradicts the assumption on the order of $\gamma$, so our claim must indeed hold.

On the other hand, trivially no sparse cell can be friends with more than $3^d-1$ cells which lie in distinct components of $\Gamma$.
Now, a simple double counting argument guarantees that the number of components of $\Gamma$ with fewer than $n^{1/d}/(4K^{1/d})$ cells is at most $2^d\nume^{-K/2}n\leq \nume^{-K/3}n/2$\COMMENT{Consider a bipartite graph with vertex sets the sparse cells and the components of $\Gamma$, where two are connected if they are friends of each other.
By double counting the number of edges, it follows that the number of components with fewer than $s^{-d}/2$ cells is at most
\[\frac{3^d-1}{2^d-1}\nume^{-K/2}n\leq2^d\nume^{-K/2}n<\nume^{-K/3}n/2,\]
where the first inequality holds since $d\geq1$ and the last inequality holds since $1/K\ll1/d$.
Since there is at most one more component, we are done.}.
Together with the bound in the first paragraph, this completes the proof.
\end{claimproof}

% We are now going to construct a Hamilton cycle in $H\cup G$.
% Roughly speaking, for each connected component of $\Gamma$, we will find a cycle in $G$ spanning all vertices which lie in the cells of this component.
% Then, by using some edges of $H$, we will incorporate all leftover vertices into said cycles and combine the cycles into a single spanning cycle.

With the statements of \cref{claim:fewsparse,claim:fewcomponents}, we can already begin the main part of the proof.
We begin by setting up some notation.
Let $\mathcal{C}^0\subseteq\mathcal{C}_{\mathrm{s}}$ be the set of all cells which contain no vertices of $V(H)$, and let $\mathcal{C}^1\subseteq\mathcal{C}_{\mathrm{s}}$ be the set of all cells which contain exactly one vertex of $V(H)$.
Let $\mathcal{C}^2\coloneqq\mathcal{C}_{\mathrm{s}}\setminus(\mathcal{C}^0\cup\mathcal{C}^1)$, and $\mathcal{C}^*\coloneqq\mathcal{C}^1\cup\mathcal{C}^2$.
Recall that, for any $v\in V(H)$, we use $v$ to refer both to the vertex and to its position in $[0,1]^d$.
Given any cell $c\in\mathcal{C}_{\mathrm{d}}$, let $\Gamma(c)$ be the connected component of $\Gamma$ which contains~$c$.
Let $\mathcal{F}\coloneqq\mathcal{C}_{\mathrm{s}}$ and  $\mathcal{F}^*\coloneqq\{c\in\mathcal{C}_{\mathrm{d}}:|V(\Gamma(c))|=1\}$.
By \cref{claim:fewsparse}\ref{claim:fewsparse1} and \cref{claim:fewcomponents} we have that $|\mathcal{F}\cup\mathcal{F}^*|\leq2\nume^{-K/3}n$.
Both $\mathcal{F}$ and $\mathcal{F}^*$ constitute sets of ``forbidden'' cells which we will avoid when connecting vertices from different cells via edges of~$H$.
These sets will be updated as we choose edges of $H$ to construct the Hamilton cycle.
Indeed, as we choose edges of $H$, each cell which contains a vertex incident to any of these edges will be added to $\mathcal{F}$, so that the different choices of edges to not interact with each other.
These edges of $H$ are chosen ``from a cell'', and $\mathcal{F}^*$ will contain all dense cells from which we must choose edges in a ``correct'' way (namely, ensuring that our process works), and thus will be avoided when choosing edges of $H$ from other cells.
We claim that, for the rest of the proof, we will always have $\mathcal{C}_{\mathrm{s}}\subseteq\mathcal{F}$ and
\begin{equation}\label{equa:Fbound}
    |\mathcal{F}\cup\mathcal{F}^*|\leq \nume^{-K/6}n,
\end{equation}
and assume so throughout.
This bound will follow from the fact that we update $\mathcal{F}$ and $\mathcal{F}^*$ at most $2\nume^{-K/3}n$ times, and each time the size of their union will increase by at most $3$\COMMENT{Of course, $8\nume^{-K/3}n\leq \nume^{-K/6}n$ because $K$ is large.}.

Consider an auxiliary graph $\Gamma'\coloneqq\Gamma$.
We are first going to modify this graph $\Gamma'$ into a connected graph.
We will update $\Gamma'$ in $t-1$ steps, where $t\leq \nume^{-K/3}n$ is the number of components of~$\Gamma$ (see \cref{claim:fewcomponents}).
In each of these steps, we will add exactly one edge to $\Gamma'$, connecting two of its components.
This auxiliary edge will correspond to a way in which we will later connect the cycles which we will construct in each component; we build the structure necessary for this at the same time as we update~$\Gamma'$.
Our definitions of $\mathcal{F}$ and $\mathcal{F}^*$ are crucial in guaranteeing that the upcoming process can be carried out.
In particular, $\mathcal{F}$ and $\mathcal{F}^*$ will always be disjoint, none of the components of (the current form of)~$\Gamma'$ will be contained in $\mathcal{F}$, and $\mathcal{F}^*$ will always contain at most one cell of each component of~$\Gamma'$.
Given any cell $c\in\mathcal{C}_{\mathrm{d}}$, let $\Gamma'(c)$ denote the connected component of~$\Gamma'$ which contains~$c$.
Initialise a set of vertices $F_{\mathrm{d}}$ and two sets of edges $E_{\mathrm{d}}$ and $E_{\mathrm{d}}^*$ as empty sets.
We proceed as follows.
\begin{enumerate}[label=\arabic*.,start=1]
\item \hypertarget{algostep3}{For} each $i\in[t-1]$, choose a smallest component $\gamma$ of $\Gamma'$, and choose an arbitrary cell $c\in V(\gamma)\setminus\mathcal{F}$ (which exists since $V(\gamma)\nsubseteq\mathcal{F}$).
Let $u_c$ and~$v_c$ be two arbitrary distinct vertices in $c$.
Choose an arbitrary $\{u_c,v_c\}$-linked cell $c'(c)\in\mathcal{C}\setminus(\mathcal{F}\cup\mathcal{F}^*\cup V(\gamma))$; its existence follows from \cref{claim:fewsparse}\ref{claim:fewsparse2}, \eqref{equa:Fbound} and the fact that $\gamma$ is a smallest component of~$\Gamma'$\COMMENT{The fact that $\gamma$ is a smallest component of $\Gamma'$ means that it has at most $s^{-d}/2$ cells.
Then, we are removing from consideration at most $s^{-d}/2+\nume^{-K/6}n+4\nume^{-K/4}n<3s^{-d}/4$ cells, so there are lots of cell to choose from.}.
Add the edge $\{c,c'(c)\}$ to $\Gamma'$.
Let $x_c\in N_H(u_c)\cap c'(c)$ and $y_c\in N_H(v_c)\cap c'(c)$ be two distinct vertices.
Add $u_c$, $v_c$, $x_c$ and $y_c$ to $F_{\mathrm{d}}$, add $u_cx_c$ and $v_cy_c$ to $E_{\mathrm{d}}^*$, and add $u_cv_c$ and $x_cy_c$ to $E_{\mathrm{d}}$.
Then, add $c$ and $c'(c)$ to $\mathcal{F}$ (if $c\in\mathcal{F}^*$, remove it from this set, so that $\mathcal{F}$ and $\mathcal{F}^*$ remain disjoint).
Finally, if $|V(\Gamma'(c))\setminus\mathcal{F}|=1$, add this remaining cell to $\mathcal{F}^*$.
\end{enumerate}
Each iteration of step \hyperlink{algostep3}{1} reduces the number of components by one, so it follows that, after we perform all iterations, $\Gamma'$ is connected.
Moreover, a moment of thought reveals that $\mathcal{F}^*$ must now be empty.

We next define some \emph{absorbing paths} which will be used to incorporate all vertices in sparse cells into a Hamilton cycle.
We define these iteratively in $|\mathcal{C}^*|$ steps.
We proceed as follows.
\begin{enumerate}[label=\arabic*.,start=2]
    \item \hypertarget{algostep1}{For} each $c\in\mathcal{C}^1$, let $v_c$ be the vertex contained in~$c$, and choose an arbitrary $v_c$-linked cell $c'(c)\in\mathcal{C}\setminus\mathcal{F}$; note that such a cell exists by \cref{claim:fewsparse}\ref{claim:fewsparse2} and \eqref{equa:Fbound}\COMMENT{By \eqref{equa:Fbound}, we have that $|\mathcal{C}\setminus(\mathcal{F}\cup\mathcal{F}^*)|\geq s^{-d}-\nume^{-K/6}n\geq 3s^{-d}/4$.
    By \cref{claim:fewsparse}\ref{claim:fewsparse2}, the number of $v_c$-linked cells is at least $s^{-d}-4\nume^{-\alpha K/4}n\geq 3s^{-d}/4$.
    Therefore, there must be some cell which belongs to both sets.
    This same argument will be used repeatedly, implicitly.}.
    Choose two distinct vertices $x_c,y_c\in N_H(v_c)\cap c'(c)$\COMMENT{At least two such vertices exist by definition, since $c'(c)$ is $v_c$-linked.}.
    Note that $e_c\coloneqq x_cy_c\in E(G)$, and define $P_c\coloneqq x_cv_cy_c$.
    Then, add $c'(c)$ to $\mathcal{F}$\COMMENT{The reason for this set is as follows. When choosing edges of $H$, we always choose pairs of edges such that the endpoints form an edge in the cell $c'(c)$. Now, if $c'(c)$ contained endpoints of several pairs of edges of $H$, then it could be that the edges inside the cell do not form a linear forest. In such a case, our proof would not work. To avoid this, we simply make sure that no cell is used for more than a pair of edges of $H$ joining it to other cells.}.
    %Moreover, if $|V(\Gamma(c'(c)))\setminus\mathcal{F}|=1$, add this remaining cell to $\mathcal{F}^*$\COMMENT{This is done to ensure that we connect everything.
    %Indeed, if we do not avoid this cell in the future, we might get stuck with a component which only contains e.g. two dense cells.
    %This is one of the primary reasons to work with $\mathcal{F}^*$.}.
    \item\hypertarget{algostep2}{For} each $c\in\mathcal{C}^2$, let~$u_c$ and~$v_c$ be two arbitrary distinct vertices in~$c$ and choose an arbitrary $\{u_c,v_c\}$-linked cell $c'(c)\in\mathcal{C}\setminus\mathcal{F}$, which again must exist by \cref{claim:fewsparse}\ref{claim:fewsparse2} and \eqref{equa:Fbound}.
    Let \mbox{$x_c\in N_H(u_c)\cap c'(c)$} and $y_c\in N_H(v_c)\cap c'(c)$ be two distinct vertices\COMMENT{so that $u_cx_c,v_cy_c\in E(H)$, which exist by the definition of $\{u_c,v_c\}$-linked cell}, and let $P_c'$ be any $(u_c,v_c)$-path which contains all vertices in~$c$ (recall that such a path exists because $G[V(H)\cap c]$ is a complete graph).
    Let $e_c\coloneqq x_cy_c\in E(G)$ and $P_c\coloneqq x_cP_c'y_c$.
    Then, add $c'(c)$ to $\mathcal{F}$.
    %Moreover, if $|V(\Gamma(c'(c)))\setminus\mathcal{F}|=1$, add this remaining cell to $\mathcal{F}^*$.
\end{enumerate}
Once this process is finished, let $F_{\mathrm{s}}\coloneqq\{x_c,y_c:c\in\mathcal{C}^*\}$ and $E_{\mathrm{s}}\coloneqq\{e_c:c\in\mathcal{C}^*\}$.

Let $F\coloneqq F_{\mathrm{s}}\cup F_{\mathrm{d}}$ and $E\coloneqq E_{\mathrm{s}}\cup E_{\mathrm{d}}$.
Note that, by construction, for each $c\in\mathcal{C}$ we have that $|F\cap c|\leq2$ and $|E(G[V(H)\cap c])\cap E|\leq1$.
We are now ready to construct the Hamilton cycle.
The main step for this is to construct a cycle in each component $\gamma$ of $\Gamma$.
We make sure that these cycles contain all edges of $E$ spanned by the vertices in the cells of the corresponding component.
For each component $\gamma$ of $\Gamma$, we proceed as follows.
\begin{enumerate}[label=\arabic*.,start=4]
\item \hypertarget{algostep4}{Let} $T$ be a spanning tree of $\gamma$.
In particular, $\Delta(T)<3^d$.
Consider an arbitrary traversal of~$T$ which, starting at a given cell, goes through every edge of $T$ twice and ends in the starting cell (this can be given, e.g., by a DFS on $T$ taking any cell $c_0$ as a root).
This traversal takes $m\coloneqq2(|V(\gamma)|-1)$ steps, each step corresponding to an edge of $T$.
We use this traversal to construct a cycle $\mathfrak{C}(\gamma)$ as follows.

Assume the traversal starts in a given cell $c_0$.
Choose a vertex $x_0\in (V(H)\cap c_0)\setminus F$ and let $P_0\coloneqq x_0$; this will be the beginning of a path which we will grow into $\mathfrak{C}(\gamma)$.
For notational purposes, set $V(P_{-1})\coloneqq\varnothing$.
For each $i\in[m]$ we define a path $P_i$ as follows.
Let~$c$ be the current cell in our traversal, and let $x_{i-1}\in (V(H)\cap c)\setminus F$ be the last vertex of $P_{i-1}$.
Let $c'$ be the next cell of the traversal.
Because $c$ and $c'$ are friends, every vertex in $c'$ is joined to every vertex in $c$ by an edge of $G$.
Choose an arbitrary vertex $x_i\in (V(H)\cap c')\setminus(F\cup V(P_{i-1}))$.
If this is the last time that~$c$ is visited in the traversal of~$T$, let $P_i'$ be any path with vertex set $(V(H)\cap c)\setminus V(P_{i-2})$ having $x_{i-1}$ as an endpoint and such that, if the vertices in~$c$ span some edge $e\in E$, then $e\in E(P_i')$, and let $P_i\coloneqq P_{i-1}P_i'x_i$; otherwise, simply let $P_i\coloneqq P_{i-1}x_i$.
To complete the cycle, let $P'$ be an $(x_0,x_m)$-path whose internal vertices are all vertices of $(V(H)\cap c_0)\setminus V(P_{m})$ and such that, if the vertices in~$c_0$ span some edge $e\in E$, then $e\in E(P')$.
We then set $\mathfrak{C}(\gamma)\coloneqq P_m\cup P'$.
Observe that every cell contains at least~$R=2\cdot3^d$ vertices and is visited at most $3^d$ times throughout the traversal; this, together with the fact that no cell spans more than one of the edges of $E$, guarantees that the choices of vertices described throughout the process can always be carried out.
\end{enumerate}

Let $\mathfrak{C}$ denote the graph which is the union of all the cycles constructed in step~\hyperlink{algostep4}{4}.
In particular, $E\subseteq E(\mathfrak{C})$.
We can now combine the cycles into a single cycle spanning all vertices in cells of $\mathcal{C}_{\mathrm{d}}$ by letting $\mathfrak{C}'\coloneqq(\mathfrak{C}\setminus E_{\mathrm{d}})\cup E_{\mathrm{d}}^*$.
In order to complete the proof, for each $c\in\mathcal{C}^*$, replace the edge $e_c\in E(\mathfrak{C}')$ by $P_c$.
\hfill\qed{}

\begin{nocremark}
By using results from percolation theory, for $d\geq2$ one can show that a.a.s.~the graph $\Gamma$ actually contains a ``giant'' component which contains, say, more than $9/10$ of the cells.
This fact can be used instead of \cref{claim:fewcomponents} to streamline our proof, simplifying some of its technical details, such as the need for the set $\mathcal{F}^*$.
Indeed, step~\hyperlink{algostep3}{1} can be avoided entirely: for every cell $c$ which does not lie in this giant component we may find a suitable $\{u,v\}$-linked cell which does, for some $u,v\in V(H)\cap c$, as in steps~\hyperlink{algostep1}{2} and~\hyperlink{algostep2}{3}.
Step~\hyperlink{algostep4}{4} can then be applied only to this giant component.
This approach, however, does not work when $d=1$, as $\Gamma$ will a.a.s.~not contain any component of linear size.
\end{nocremark}

\section{Final remarks}\label{section4}

\Cref{thm:main} is asymptotically best possible in the sense that, for each $\alpha\in(0,1/2)$, if we let $r=(c/n)^{1/d}$ for a sufficiently small constant $c$, then there exist graphs $H$ with minimum degree $\alpha n$ such that $H\cup G^d(n,r)$ is a.a.s.~not Hamiltonian.
Indeed, let $H$ be a complete unbalanced bipartite graph with parts $A$ and $B$ of sizes $\alpha n$ and $(1-\alpha)n$, respectively.
Clearly, if $G^d(n,r)[B]$ contains more than $\alpha n$ isolated vertices, then $H\cup G^d(n,r)$ cannot contain a Hamilton cycle.
Our claim thus follows immediately from the following lemma (where we also allow $\alpha$ to depend on $n$).

\begin{lemma}\label{lem:counterexample}
Let $d\geq1$ be an integer, and let $V_d$ be the volume of the ball of radius $1$ in $d$ dimensions.
If $1/2>\alpha=\alpha(n)=\omega(n^{-1/2})$, then a.a.s.~$G^d(n,(-\log(3\alpha/2)/(2V_dn))^{1/d})$ contains at least $4\alpha n/3$ isolated vertices.
\end{lemma}

\begin{proof}
Let $f\colon([0,1]^d)^n\to\mathbb{Z}$ be a function that, given a set of $n$ points $x_1,\ldots,x_n\in[0,1]^d$, returns the number of isolated vertices in the geometric graph constructed on these points with radius $(-\log(3\alpha/2)/(2V_dn))^{1/d}$.
One can readily check that $f$ is an $L$-Lipschitz function, for some $L=L(d)$\COMMENT{To see this, consider a ball of radius $r$, and partition it according to coordinate hyperplanes into (subsets of) cells of side $r/\sqrt{d}$.
This ensures that, if any two points lie in the same cell, then they form an edge.
Now, this clearly partitions the ball into at most $(2\sqrt{d})^d$ cells, so the pigeonhole principle ensures that there cannot be more than $(2\sqrt{d})^d$ vertices in the ball which are isolated.
Then, clearly, putting a new vertex in the center of this ball will make all these points become non-isolated, and it can make the same number of points become isolated if removed from the centre and placed elsewhere.
Since changing the value of one $X_i$ cannot change the number of isolated vertices by more than that many, it follows that $X$ is $(2\sqrt{d})^d$-Lipschitz.
(Note that it is certainly $L$-Lipschitz for a smaller $L$, which we do not try to optimise.)}.

Fix an arbitrary $i\in[n]$.
Let $X_1,\ldots,X_n$ be independent uniform random variables on $[0,1]^d$.
We are first going to compute the probability that $i$ is isolated in $G=G^d(n,(-\log(3\alpha/2)/(2V_dn))^{1/d})$.
This is the probability that no $X_j$ with $j\neq i$ is assigned to a position within distance $r$ of $X_i$, that is\COMMENT{To prove the second inequality, we use the fact that, if $x>0$ is sufficiently close to $0$, then $\nume^{-x}\geq 1-x\geq \nume^{-2x}$, to obtain
\[(1+\log(3\alpha/2)/(2n))^{n-1}\geq(\nume^{\log(3\alpha/2)/n})^{n-1}\geq\nume^{\log(3\alpha/2)},\]
where the last inequality holds since $\log(3\alpha/2)<0$.},
\[\mathbb{P}[i\text{ is isolated}]\geq(1+\log(3\alpha/2)/(2n))^{n-1}\geq \nume^{\log(3\alpha/2)}=3\alpha/2.\]

Let $X\coloneqq f(X_1,\ldots,X_n)$ denote the number of isolated vertices in $G$, so $\mathbb{E}[X]\geq3\alpha n/2$.
Since $f$ is an $L$-Lipschitz function, by \cref{lem:Azuma} we conclude that\COMMENT{We have that
\[\mathbb{P}[X\leq\mathbb{E}[X]-\alpha n]\leq \nume^{-\frac{\alpha^2n^2}{72nL^2}}=\nume^{-\alpha^2n/(72L^2)}.\]}
\[\mathbb{P}[X\leq4\alpha n/3]\leq\mathbb{P}[X\leq\mathbb{E}[X]-\alpha n/6]\leq \nume^{-\Theta(\alpha^2n)}=o(1),\]
where the last equality relies on the lower bound on $\alpha$.
\end{proof}

A natural questions arises from \cref{thm:main}: what can we say about Hamiltonicity of randomly perturbed graphs when $\alpha$ is allowed to be a function of $n$ (which tends to $0$)?
This same question was recently considered by \citet{HMMMP21} when the random graph $G$ is binomial, and they extended the result of \citet{BFM03} by proving that $G=G(n,\Theta(-\log\alpha/n))$ is both sufficient and necessary.
In view of \cref{lem:counterexample}, the following seems a natural conjecture (and, if true, it would be best possible up to the value of the constant $C$).

\begin{conjecture}
For every integer $d\geq1$, there exists a constant $C$ such that the following holds.
Let $\alpha=\alpha(n)\in(0,1/2)$, and let $H$ be an $n$-vertex graph with minimum degree at least $\alpha n$.
Then, a.a.s. $H\cup G^d(n,(-C\log\alpha/n)^{1/d})$ is Hamiltonian.
\end{conjecture}

% Note that, with our current approach, if, say, $\alpha=\Theta(n^{-1/3})$, we run into trouble when trying to prove \cref{claim:fewsparse}\ref{claim:fewsparse2}.
% Indeed, here for each pair of vertices $u,v$ we have that almost all cells are $\{u,v\}$-unlinked.
% While it still seems plausible that there could be sufficiently many $\{u,v\}$-linked cells (given that the number of sparse cells is much smaller than in our case), this would require a much more careful analysis than we have here.

We also wish to remark upon two features of our proof of \cref{thm:main}.
First, the proof is constructive, meaning that it provides an algorithm to find Hamilton cycles in $H\cup G^d(n,r)$.
In particular, if the properties of \cref{claim:fewsparse} hold (which occurs a.a.s.), it provides a deterministic algorithm that outputs a Hamilton cycle in $H\cup G^d(n,r)$.
Furthermore, observe that, throughout the proof, we actually do not need \cref{claim:fewsparse}\ref{claim:fewsparse2} to hold for every pair of vertices, but only for those that we pick throughout the process, which are only linearly many.
This means that the properties of \cref{claim:fewsparse} can be checked in $\bigO(n^2)$ time, which is linear in the size of $H\cup G^d(n,r)$.
Then, the construction of the Hamilton cycle also takes $\bigO(n^2)$ time.
This follows directly from the proof, and can be checked by retracing the steps.

Second, all throughout the paper we have considered the $\ell_2$ Euclidean norm for simplicity.
Our proof generalises directly to all $\ell_p$ norms, $1\leq p\leq\infty$, by adjusting some of the constants.

Furthermore, we note that, under the same conditions as in the statement of \cref{thm:main}, we can actually show that $H\cup G^d(n,r)$ is pancyclic.
Indeed, our proof can easily be modified for this.
From the Hamilton cycle that we construct, given that it contains many subpaths whose vertices actually form cliques in $G^d(n,r)$, one can iteratively reduce the number of vertices in the cycle.
This can be balanced with also removing some of the paths which correspond to sparse cells, as well as leaves of the auxiliary graph $\Gamma'$, to prove that cycles of all lengths can be constructed.

Finally, we remark that our result opens the door to questions about other spanning structures in dense graphs perturbed by random geometric graphs.
In this direction, \citet{EH22} very recently generalised \cref{thm:main} to deal with powers of Hamilton cycles, showing that, for $k$ independent of $n$, the required radius for the a.a.s.\ containment of the $k$-th power of a Hamilton cycle in this setting is the same (up to constant factors depending on $k$) as that required for Hamiltonicity.
This has a number of direct consequences, such as for questions related to $F$-factors or $2$-universality (which directly implies pancyclicity).

\COMMENT{A simple observation is that our results also hold in the torus. This is trivial since the random geometric graph defined on the torus is always a supergraph of the same geometric graph on the hypercube.}

\section*{Acknowledgement}

I would like to thank Xavier Pérez-Giménez for some very helpful discussions about the topic of this paper.
I am also indebted to anonymous referees for their helpful remarks.

% Use with natbib, not biblatex:
\bibliographystyle{mystyle} 
\bibliography{geopert}

\begin{thebibliography}{25}
\newcommand{\enquotenew}[1]{`#1'}
\providecommand{\natexlab}[1]{#1}
\providecommand{\url}[1]{\texttt{#1}}
\providecommand{\urlprefix}{URL }
\providecommand{\doi}[1]{\textsc{doi}:
  \href{https://doi.org/#1}{\nolinkurl{#1}}}
\providecommand*{\eprint}[2][]{arXiv:
  \href{https://arxiv.org/abs/#2}{\nolinkurl{#2}}}

\bibitem[{{Antoniuk}, {Dudek}, {Reiher}, {Ruci{\'n}ski} and
  {Schacht}(2021)}]{ADRRS21}
S.~{Antoniuk}, A.~{Dudek}, C.~{Reiher}, A.~{Ruci{\'n}ski} and M.~{Schacht},
  {High powers of Hamiltonian cycles in randomly augmented graphs}.
  \emph{Journal of Graph Theory} 98.2 (2021),  255--284,
  \doi{10.1002/jgt.22691}.

\bibitem[{Balogh, Bollobás, Krivelevich, Müller and Walters(2011)}]{BBKMW11}
J.~Balogh, B.~Bollobás, M.~Krivelevich, T.~Müller and M.~Walters, {Hamilton
  cycles in random geometric graphs}. \emph{Ann. Appl. Probab.} 21 (2011),
  1053--1072, \doi{10.1214/10-AAP718}.

\bibitem[{Balogh, Treglown and Wagner(2019)}]{BTW19}
J.~Balogh, A.~Treglown and A.~Z. Wagner, {Tilings in randomly perturbed dense
  graphs}. \emph{Combin. Probab. Comput.} 28 (2019),  159--176,
  \doi{10.1017/S0963548318000366}.

\bibitem[{Bohman, Frieze and Martin(2003)}]{BFM03}
T.~Bohman, A.~Frieze and R.~Martin, {How many random edges make a dense graph
  {H}amiltonian?} \emph{Random Structures Algorithms} 22 (2003),  33--42,
  \doi{10.1002/rsa.10070}.

\bibitem[{B\"{o}ttcher, Han, Kohayakawa, Montgomery, Parczyk and
  Person(2019)}]{BHKMPP19}
J.~B\"{o}ttcher, J.~Han, Y.~Kohayakawa, R.~Montgomery, O.~Parczyk and
  Y.~Person, {Universality for bounded degree spanning trees in randomly
  perturbed graphs}. \emph{Random Structures Algorithms} 55 (2019),  854--864,
  \doi{10.1002/rsa.20850}.

\bibitem[{B\"{o}ttcher, Montgomery, Parczyk and Person(2020)}]{BMPP20}
J.~B\"{o}ttcher, R.~Montgomery, O.~Parczyk and Y.~Person, {Embedding spanning
  bounded degree graphs in randomly perturbed graphs}. \emph{Mathematika} 66
  (2020),  422--447, \doi{10.1112/mtk.12005}.

\bibitem[{{B{\"o}ttcher}, {Parczyk}, {Sgueglia} and {Skokan}(2020)}]{BPSS20}
J.~{B{\"o}ttcher}, O.~{Parczyk}, A.~{Sgueglia} and J.~{Skokan}, {{Triangles in
  randomly perturbed graphs}}. \emph{arXiv e-prints}  (2020).
  \eprint{2011.07612}.

\bibitem[{{B{\"o}ttcher}, {Parczyk}, {Sgueglia} and {Skokan}(2021)}]{BPSS21}
---{}---{}---, {{Cycle factors in randomly perturbed graphs}}. \emph{arXiv
  e-prints}  (2021). \eprint{2103.06136}.

\bibitem[{{Condon}, {Espuny D\'iaz}, {Gir{\~a}o}, {K{\"u}hn} and
  {Osthus}(2020)}]{CEGKO20}
P.~{Condon}, A.~{Espuny D\'iaz}, A.~{Gir{\~a}o}, D.~{K{\"u}hn} and D.~{Osthus},
  {Hamiltonicity of random subgraphs of the hypercube}. \emph{arXiv e-prints}
  (2020). \eprint{2007.02891}.

\bibitem[{Dudek, Reiher, Ruci\'{n}ski and Schacht(2020)}]{DRRS20}
A.~Dudek, C.~Reiher, A.~Ruci\'{n}ski and M.~Schacht, {Powers of {H}amiltonian
  cycles in randomly augmented graphs}. \emph{Random Structures Algorithms} 56
  (2020),  122--141, \doi{10.1002/rsa.20870}.

\bibitem[{Díaz, Mitsche and Pérez(2007)}]{DMP07}
J.~Díaz, D.~Mitsche and X.~Pérez, {Sharp Threshold for Hamiltonicity of
  Random Geometric Graphs}. \emph{SIAM J. Discrete Math.} 21 (2007),  57--65,
  \doi{10.1137/060665300}.

\bibitem[{{Espuny D{\'\i}az} and {Gir{\~a}o}(2021)}]{EG21}
A.~{Espuny D{\'\i}az} and A.~{Gir{\~a}o}, {{Hamiltonicity of graphs perturbed
  by a random regular graph}}. \emph{arXiv e-prints}  (2021).
  \eprint{2101.06689}.

\bibitem[{{Espuny D{\'\i}az} and {Hyde}(2022)}]{EH22}
A.~{Espuny D{\'\i}az} and J.~{Hyde}, {{Powers of Hamilton cycles in dense
  graphs perturbed by a random geometric graph}}. \emph{arXiv e-prints}
  (2022). \eprint{2205.08971}.

\bibitem[{{Hahn-Klimroth}, {Maesaka}, {Mogge}, {Mohr} and
  {Parczyk}(2021)}]{HMMMP21}
M.~{Hahn-Klimroth}, G.~S. {Maesaka}, Y.~{Mogge}, S.~{Mohr} and O.~{Parczyk},
  {{Random perturbation of sparse graphs}}. \emph{{Electron. J. Comb.}} 28.2
  (2021),  research paper p2.26, 12, \doi{10.37236/9510}.

\bibitem[{{Han}, {Morris} and {Treglown}(2020)}]{HMT19+}
J.~{Han}, P.~{Morris} and A.~{Treglown}, {{Tilings in randomly perturbed
  graphs: bridging the gap between Hajnal-Szemer{\'e}di and
  Johansson-Kahn-Vu}}. \emph{Random Structures Algorithms}  (2020),  1--37,
  \doi{10.1002/rsa.20981}.

\bibitem[{Han and Zhao(2020)}]{HZ20}
J.~Han and Y.~Zhao, {Hamiltonicity in randomly perturbed hypergraphs}. \emph{J.
  Combin. Theory Ser. B} 144 (2020),  14--31, \doi{10.1016/j.jctb.2019.12.005}.

\bibitem[{Janson, \L{}uczak and Ruci\'{n}ski(2000)}]{JLR}
S.~Janson, T.~\L{}uczak and A.~Ruci\'{n}ski, \emph{Random graphs}.
  Wiley-Interscience Series in Discrete Mathematics and Optimization,
  Wiley-Interscience, New York (2000), \doi{10.1002/9781118032718}.

\bibitem[{Joos and Kim(2020)}]{JK19}
F.~Joos and J.~Kim, {Spanning trees in randomly perturbed graphs}. \emph{Random
  Structures Algorithms} 56 (2020),  169--219, \doi{10.1002/rsa.20886}.

\bibitem[{Krivelevich, Kwan and Sudakov(2016)}]{KKS16}
M.~Krivelevich, M.~Kwan and B.~Sudakov, {Cycles and matchings in randomly
  perturbed digraphs and hypergraphs}. \emph{Combin. Probab. Comput.} 25
  (2016),  909--927, \doi{10.1017/S0963548316000079}.

\bibitem[{Krivelevich, Kwan and Sudakov(2017)}]{KKS17}
---{}---{}---, {Bounded-degree spanning trees in randomly perturbed graphs}.
  \emph{SIAM J. Discrete Math.} 31 (2017),  155--171, \doi{10.1137/15M1032910}.

\bibitem[{McDowell and Mycroft(2018)}]{MM18}
A.~McDowell and R.~Mycroft, {Hamilton {$\ell$}-cycles in randomly perturbed
  hypergraphs}. \emph{Electron. J. Combin.} 25 (2018),  Paper No. 4.36, 30,
  \doi{10.37236/7671}.

\bibitem[{Müller, Pérez-Giménez and Wormald(2011)}]{MPW11}
T.~Müller, X.~Pérez-Giménez and N.~Wormald, {Disjoint Hamilton cycles in the
  random geometric graph}. \emph{J. Graph Theory} 68 (2011),  299--322,
  \doi{10.1002/jgt.20560}.

\bibitem[{{Nenadov} and {Truji\'c}(2021)}]{NT21}
R.~{Nenadov} and M.~{Truji\'c}, {{Sprinkling a few random edges doubles the
  power}}. \emph{{SIAM J. Discrete Math.}} 35.2 (2021),  988--1004,
  \doi{10.1137/19M125412X}.

\bibitem[{{Penrose}(2003)}]{Pen03}
M.~{Penrose}, \emph{{Random geometric graphs}}, \emph{{Oxf. Stud. Probab.}},
  vol.~5. Oxford: Oxford University Press (2003).

\bibitem[{{Penrose}(2016)}]{Pen16}
---{}---{}---, {{Lectures on random geometric graphs}}. \emph{{Random graphs,
  geometry and asymptotic structure}},  67--101, Cambridge: Cambridge
  University Press (2016).

\end{thebibliography}

% Use with biblatex, not natbib:
% \printbibliography

\end{document}